\documentclass[11pt]{amsart}
\input{macros}
\title{The finite cohesiveness principle}
\author{Mengzhou Sun}
\address{Department of Mathematics, National University of Singapore, Singapore 119076}
\email{sunm07@u.nus.edu}
\address{Institute of Mathematics, University of Warsaw, 
Banacha 2, 02-097 Warszawa, Poland}
\email{m.sun3@uw.edu.pl}
\begin{document}

\begin{abstract}
    We investigate the logical strength of the cohesiveness principle when restricted to finite sequences of sets, denoted by $\fincoh$, over different base theories.
    Our main result shows that $\fincoh$ entails $\ind\Sigma_1^0$ over the weaker base theory $\RCA_0^*$, thereby answering a question posed by Fiori-Carones, Ko\l{}odziejczyk and Kowalik in~\cite{art:weakercousins}.
    In addition, we show that $\fincoh$ is not provable over $\WKL_0$.
    
\end{abstract}
\maketitle
The \defm{cohesiveness principle} ($\coh$) is a combinatorial statement asserting that, for any given sequence of sets, there exists an infinite set that cannot be partitioned into two infinite subsets by any member of the sequence.
This principle plays a central role in the study of reverse mathematics of Ramsey's Theorem for Pairs ($\RT_2^2$).    
Cholak, Jockusch, and Slaman~\cite{art:CJSramsey} showed that $\mathnoun{RT}^2_2$ decomposes\footnote{The original proof of $\RT_2^2\vdash\coh$ presented in \cite{art:CJSramsey} requires $\ind\Sigma_2^0$. A proof that can be formalized within $\RCA_0$ was later provided by Mileti\cite{MiletiPhD}.} into $\coh$ and Stable Ramsey's Theorem for Pairs ($\mathnoun{SRT}^2_2$) over $\RCA_0$, and that $\coh$ is $\Pi_1^1$-conservative over $\RCA_0$.

More recently, a series of papers~\cite{art:weakercousins, art:KKY, art:y_ramsey_rca*} by Fiori-Carones, Ko\l{}odziejczyk, Kowalik and Yokoyama systematically studied the strength of combinatorial principles related to $\RT^2_2$ over the weaker base theory $\RCA_0^*$.
It is shown in \cite{art:weakercousins} that most such principles, including $\RT_2^2$, $\SRT_2^2$ and $\CRT_2^2$ (Cohesive Ramsey's Theorem for Pairs), hold in a model of $\RCA_0^*+\neg\ind\Sigma_1^0$ if and only if they hold on a $\Sigma_1^0$-definable cut.
In particular, most of these principles are consistent with, and in fact $\forall\Pi_3^0$ conservative over $\RCA_0^*+\neg\ind\Sigma_1^0$.

However, the principle $\coh$ is an exception to this pattern, partly due to the additional first-order quantifier complexity of $\coh$ compared to other principles.
Accordingly, Fiori-Carones et al. posed the question~\cite[Question 5.8]{art:weakercousins} of whether $\coh$ is consistent with $\RCA_0^*+\neg\ind\Sigma_1^0$.
It is also worth mentioning that this question traces back to an earlier preprint of Belanger~\cite{Belanger}, in which he asked if $\coh$ is $\Pi_1^1$-conservative over $\RCA_0^*$.
Belanger's question was answered negatively in \cite{art:weakercousins}.

Perhaps surprisingly, it turns out that the principle $\coh$ restricted to finite sequences of sets, denoted by $\fincoh$ in this paper, already implies $\ind\Sigma_1^0$ over $\RCA_0^*$.
This serves as the original motivation for the present work.

The paper is organized as follows.
Section 1 reviews some basic notation and background on first- and second-order arithmetic and formally defines the principle $\fincoh$.
In Section 2, we present the proof that $\RCA_0^*+\fincoh$ implies $\ind\Sigma_1^0$.
In Section 3, we investigate the strength of $\fincoh$ over the usual base theory $\RCA_0$, showing that $\fincoh$ is provable over $\RCA_0+\bd\Sigma_2^0$, but not over $\WKL_0$. 
The paper concludes with some open questions and observations concerning these questions.

\section{Preliminaries}
We assume some basic familiarity with first- and second-order arithmetic~\cite{book:HP, book:simpson}.
The symbol $\omega$ denotes the set of standard natural numbers, whereas $\IN$ denotes the set of natural numbers as formalized within an arithmetic theory.
We write $\Sigma_n^0$, $\Pi_n^0$, $\Delta_n^0$ for the usual second-order formula classes defined by counting alternating blocks of first-order quantifiers; second-order free variables are permitted, while second-order quantifiers are excluded.
The classes $\Sigma_n$, $\Pi_n$ and $\Delta_n$ are their first-order counterparts, respectively.
For any set $A$, we write $\Sigma_n(A)$ to denote the class of $\Sigma_n^0$ formulas in which $A$ is the only second-order parameter.

The schemes $\ind\Sigma_n^0$ and $\bd\Sigma_n^0$ denote, respectively, the induction and collection schemes for $\Sigma_n^0$ formulas, and $\ind\Sigma_n$ and $\bd\Sigma_n$ are defined in a similar way.
The theory $\RCA_0$ consists of Robinson arithmetic, $\ind\Sigma_1^0$ and comprehension for $\Delta_1^0$ formulas. 
The weaker theory $\RCA_0^*$ is obtained from $\RCA_0$ by replacing $\ind\Sigma_1^0$ with $\bd\Sigma_1^0$ and adding a statement $\exp$ that guarantees the totality of exponentiation.
The theory $\WKL_0$ extends $\RCA_0$ with an axiom stating that each infinite binary tree has an infinite path.
 
For any sets $A$, $B$, the notation $A\subseteq^*B$ abbreviates $\ex b\falg x b (x\in A\rightarrow x\in B$);
the notation $A\subseteq_\cf B$ abbreviates $A\subseteq B\wedge\fain b B\exin a A a>b$.
We say a set $A$ is \defm{infinite} or \defm{unbounded} if $A\subseteq_\cf \IN$.
We write $A^c$ to denote the complement of $A$.
For any $a<b$ in $\IN$, the ``interval'' $[a,b)$ denotes the set $\{a, a+1, \dots,b-1\}$. 

We adopt the standard pairing function, where the code of an ordered pair $(a,b)$ under this pairing function is denoted by $\langle a,b\rangle$.
With such a pairing function, any set $A$ can be viewed as a sequence of sets $(A_i)_{i\in \IN}$, where $A_i = \{x\in \IN\mid \langle i,x\rangle\in A\}$.
Let $M$ be a model of $\ind\Delta_0+\exp$.
For any element $c$ in $M$, we identify $c$ with a subset of $M$, which is denoted by $\Ack c$, by defining $x\in \Ack c$ to mean the $x$-th digit in the binary expansion of~$c$ is~$1$.
Let $I$ be a proper cut of $M$.
For any set $X\subseteq I$, we say that $X$ is \defm{coded in} $I$, if there is some $c\in M$ such that $X = \Ack c\cap I$.
We define $\cod M I = \{\Ack c\cap I \mid c\in M\}$.

Formally, the cohesiveness principle is defined as follows:

\begin{description}
    \item[$\coh$]  Let $R=(R_k)_{k\in\IN}$ be an infinite sequence of sets. Then there is an infinite set $H$ such that for any $k\in\IN$, either $H\subseteq^*R_k$ or $H\subseteq^*R_k^c$ holds.
\end{description}
The finite cohesiveness principle is defined by replacing $(R_k)_{k\in\IN}$ with a sequence of finite length.
\begin{description}
    \item [$\fincoh$] Let $\ell\in\IN$ and $R=(R_k)_{k<\ell}$ be a sequence of sets of length $\ell$.
    Then there is an infinite set $H$ such that for any $k<\ell$, either $H\subseteq^*R_k$ or $H\subseteq^*R_k^c$ holds.
\end{description}
In both definitions, we say that $H$ is \defm{$R$-cohesive}.

\section{The cohesiveness principle over a weaker base theory}
To establish our main result, we construct a specific instance of $\fincoh$ such that the finite sequence of sets encodes all coded subsets of a $\Sigma_1^0$-definable cut $I$.
With the help of Lemma~\ref{lem:cmcoding}, one can show that the resulting cohesive set can be decoded as a coded subset of $I$, which turns out to be cohesive for all coded subsets of $I$. 
This yields a contradiction.

The following two lemmas concerning the failure of $\ind\Sigma_1^0$ and the coding ability of $\bd\Sigma_1^0$ will be used essentially in the proof.
We state them explicitly for the reader's convenience.
\begin{lem}[folklore]\label{lem:cffun}
    Let $(M,\mathcal{X})$ be a model of $\RCA_0^*+\neg\ind\Sigma_1^0$, and let $I\subseteq M$ be a $\Sigma_1^0$-definable proper cut.
    Then there exists a non-decreasing function $f\colon I\rightarrow M$ in $\mathcal{X}$, whose range is cofinal in $M$.
\end{lem}
\begin{lem}[Chong--Mourad Coding Lemma~\cite{art:CM}]\label{lem:cmcoding}
    Let $(M,\mathcal{X})$ be a model of $\RCA_0^*$, and let $I\subseteq M$ be a $\Sigma_1^0$-definable cut.
    For any $X\subseteq M$, if $X\cap I$ and $X^c\cap I$ are both $\Sigma_1^0$-definable in $(M,\mathcal{X})$, then $X\cap I$ is coded in $I$.
\end{lem}
\begin{thm}\label{thm:maincoh}
    $\RCA_0^*+\fincoh\vdash\ind\Sigma_1^0$.
\end{thm}
\begin{proof}
    Suppose not, then there is a model $(M,\mc{X})\models \RCA_0^*+\neg\ind\Sigma_1^0+\fincoh$.
    Let $I$ be a $\Sigma_1^0$-definable proper cut of $M$.
    By Lemma~\ref{lem:cffun}, there is some non-decreasing function $f\colon I\rightarrow M$ in $\mathcal{X}$ whose range is cofinal in $M$.
    Fixing some $b>I$, let $R\subseteq M$ be such that for any $k<2^b$,
    \[\langle k,y\rangle\in R\iff\ex x (x\in{\Ack k\cap I}\wedge y\in [f(x),f(x+1))).\]
    Since 
    \[\langle k,y\rangle\notin R\iff\ex x (x\in {\Ack k^c\cap I}\wedge y\in [f(x),f(x+1))),\]
    $R$ is $\Delta_1^0$-definable in $(M,\mathcal{X})$ and thus $R\in\mathcal{X}$ is justified.
    We view $R$ as a sequence of sets $(R_k)_{k<2^b}$.
    By $\fincoh$, there exists some unbounded set $H\in\mc{X}$ that is $(R_k)_{k<2^b}$-cohesive.
    Consider the set $X\subseteq I$ defined by
    \[x\in X\iff x\in I\wedge\ex y (y\in{[f(x),f(x+1))\cap H}).\]
    Meanwhile, we have
    \[x\in X^c\cap I\iff x\in I\wedge\fa y (y\notin{[f(x),f(x+1))}\cap H).\]
    Note that the definition of $X^c\cap I$ above is actually $\Sigma_1^0$, as it is equivalent to 
    \[x\in I\wedge\ex {z,w}(f(x)=z\wedge f(x+1)=w\wedge \falt y w y\notin{[z,w)}\cap H).\]
    Thus, the set $X$ is coded in $I$ by Lemma~\ref{lem:cmcoding}.
    Notice that $X\subseteq_\cf I$ as $H\subseteq_\cf M$.
    We then find two subsets $X_1, X_2\in \cod M I$ of $X$ satisfying the following properties:
    \begin{enumerate}
        \item $X_1, X_2\subseteq_\cf I$.\label{enum:1}
        \item $X_1\cap X_2=\emptyset$.\label{enum:2}
    \end{enumerate}
    More explicitly, let $X_1,X_2\in \cod M I$ be such that for any $i\in I$,
    \begin{align*}
        i\in X_1\iff i\in X\wedge |X\restdto i| \text{ is even},\\
        i\in X_2\iff i\in X\wedge |X\restdto i| \text{ is odd},
    \end{align*} 
    where $X\restdto i=X\cap[0,i)$ and $|\cdot|$ is the cardinality function.
    Let $k_1,k_2\in M$ code $X_1, X_2$ in $I$ respectively.
    We may assume that $k_1,k_2<2^b$.
    Then, by the definition of $R$,
    \[y\in R_{k_j} \iff \exin x {X_j} (y\in [f(x),f(x+1)))\]
    for $j\in\{1,2\}$.
    Combining with the properties~(\ref{enum:1}) and (\ref{enum:2}) above, we have
    \begin{enumerate}[label=(\arabic*')]
        \item $H\cap R_{k_1}, H\cap R_{k_2}$ are both cofinal in $M$.
        \item $R_{k_1}\cap R_{k_2}= \emptyset$.
    \end{enumerate}
    So, $H\cap R_{k_1}\subseteq_\cf M$ and $H\cap R_{k_2}\subseteq H\cap R_{k_1}^c\subseteq_\cf M$, which contradicts the fact that $H$ is $(R_k)_{k<2^b}$-cohesive.
\end{proof}

\section{The strength of $\fincoh$ over $\RCA_0$}
    We now examine the strength of $\fincoh$ over the usual base theory $\RCA_0$.
    We begin by showing that $\fincoh$ is not trivial over $\RCA_0$.
    In fact, we are able to show that it is not provable over $\WKL_0$.

    To establish this result, we introduce an auxiliary principle, denoted by $\finsep$.
    The principle $\sep$ was introduced by Belanger in \cite{Belanger}, where he proved that it is equivalent to $\coh$ over $\RCA_0+\bd\Sigma_2^0$.
    The principle $\finsep$ is formulated as a finite variant of $\sep$, in direct analogy with $\fincoh$.
\begin{description}
    \item[$\sep$] \itshape For every two disjoint $\Sigma^0_2$-definable sets $A$, $B$ there exists a $\Delta^0_2$-definable set $C$ that separates $A$ and $B$, i.e., $A\subseteq C$ and $C\subseteq B^c$. 
    \item[$\finsep$] \itshape For every $b\in \IN$ and every two disjoint $\Sigma^0_2$-definable sets $A$, $B$ bounded by $b$, there exists a $\Delta^0_2$-definable set $C$ such that $A\subseteq C$ and $C\subseteq {B^c}$. 
\end{description}
\begin{lem}
    $\RCA_0+\fincoh\vdash\finsep$.
\end{lem}
\begin{proof}
    The argument here is exactly the same as the proof that $\coh$ implies $\sep$ over $\RCA_0$.
    (See \cite[Lemma 4.2]{Belanger}).
\end{proof}
\begin{thm}\label{thm:rcanotprovefinsep}
$\WKL_0\nvdash\finsep$.
\end{thm}
\begin{proof}
    Let $M\models\ind\Sigma_1$ be a pointwise $\Sigma_2$-definable model.
    (For a construction of such a model, see e.g., ~\cite[Theorem IV.1.33]{book:HP}.)
    By iteratively applying the formalized very low basis theorem within $M$ (see~\cite[Theorem I.3.8, Remark I.3.21]{book:HP}), we expand $M$ to some second-order model $(M,\mathcal{X})\models\WKL_0$, such that for each $A\in\mathcal{X}$, any $\Sigma_1(A)$-definable subset of $M$ is $\Sigma_0^\exp(\Sigma_1)$-definable\footnote{The formula class $\Sigma_0^\exp(\Sigma_1)$ is the closure of the class of $\Sigma_1$ formulas under Boolean operations and bounded quantifiers of the form $\exists {x<{t(\ov y)}}$ or $\forall {x<{t(\ov y)}}$, where $t$ is a term in which $\exp$ may occur as a function symbol.} in $M$.
    
    We claim that $(M,\mathcal{X})$ does not satisfy $\finsep$.
    Suppose otherwise; let $(\Phi_e)_{e\in M}$ be an effective list of all Turing functionals in $M$. (For Turing functionals and Turing reducibility formalized within fragments of $\PA$, see~\cite{art:CYjumpcut}.)
    Fixing any $b>\omega$ and a universal $\Sigma_1$-definable set $0'$, we consider the following $\Sigma_2$-definable sets bounded by $b$.
    \begin{align*}
        A=\{e<b\mid \Phi_e^{0'}(e)\downarrow=0\},\\
        B=\{e<b\mid \Phi_e^{0'}(e)\downarrow=1\}.
    \end{align*}
    By $\finsep$, there is a $\Delta_2^0$-definable set $C\subseteq M$ that separates $A$ and $B$.

    We now eliminate all occurrences of first- and second-order parameters in $C$.
    Without loss of generality, we may assume that
    \begin{align*}
        x\in C\iff \ex u\fa v \lambda(u,v,x,d,D),\\
        x\notin C\iff \ex u\fa v\mu(u,v,x,d,D),
    \end{align*}
    where $\lambda$ and $\mu$ are $\Delta_0^0$ formulas, and $d, D$ are the only first- and second-order parameters appearing in these formulas.
    By the construction of $\mathcal{X}$, the formula $\fa v\lambda(u,v,x,d,D)$ is equivalent to a $\Sigma_0^{\exp}(\Sigma_1)$ formula in $M$, with no second-order parameters.
    Over $\ind\Sigma_1$, every $\Sigma_0^\exp (\Sigma_1)$ formula is equivalent to some $\Delta_2$ formula~\cite[Lemma I.2.76]{book:HP}.
    It follows that both $C$ and its complement are definable by some $\Sigma_2$ formulas, say $\phi(x,e)$ and $\psi(x,e)$, respectively, where $e$ is the only parameter occuring in these definitions. 
    Since $M$ is pointwise $\Sigma_2$-definable, we may assume that $e$ is defined by a $\Sigma_2$ formula $\delta(e)$.
    Consequently, both $C$ and its complement are defined by the following parameter-free $\Sigma_2$ formulas:
    \begin{align*}
        x\in C\iff \ex e(\delta(e)\wedge \phi(x,e)),\\
        x\notin C\iff \ex e(\delta(e)\wedge \psi(x,e)).
    \end{align*}
    
    By \cite[Corollary 3.1]{art:CYjumpcut}, the set $C$ is weakly recursive in $0'$, that is, there exists some $e_0\in M$ such that 
    \begin{align*}
        x\in C\Rightarrow\Phi_{e_0}^{0'}(x)=1,\\
        x\notin C\Rightarrow\Phi_{e_0}^{0'}(x)=0.
    \end{align*}
    Moreover, the index $e_0$ can be effectively computed from the indices of the pair of $\Sigma_2$ definitions of $C$ and its complement, which are standard natural numbers.
    So, we may assume that $e_0\in\omega$.
    Then, by the definitions of $C$ and $e_0$,
    \begin{align*}
        \Phi_{e_0}^{0'}(e_0)=0\Rightarrow e_0\in A\Rightarrow e_0\in C\Rightarrow\Phi_{e_0}^{0'}(e_0)=1,\\
        \Phi_{e_0}^{0'}(e_0)=1\Rightarrow e_0\in B\Rightarrow e_0\notin C\Rightarrow\Phi_{e_0}^{0'}(e_0)=0,
    \end{align*}  
    which is impossible as $\Phi_{e_0}^{0'}$ is a total function with outputs in $\{0,1\}$.
\end{proof}
\begin{cor}\label{cor:wklnot->fincoh}
    $\WKL_0\nvdash\fincoh$.\qed
\end{cor}

On the other hand, a simple argument shows that $\bd\Sigma_2^0$ implies $\fincoh$ over $\RCA_0$.
Therefore, the ``second-order'' character of $\fincoh$ becomes apparent only in the absence of sufficiently strong induction or collection.
    
\begin{proposition}\label{prop:bdsigma2->fincoh}
    $\RCA_0+\bd\Sigma_2^0\vdash\fincoh$.
\end{proposition}
\begin{proof}
    Let $(M,\mathcal{X})\models\RCA_0+\bd\Sigma_2^0$ and $(R_k)_{k<l}$ be a sequence of sets in $\mathcal{X}$.
    For any binary string $\sigma$ of length $l$ in $M$, define 
    \[R_\sigma=\bigcap_{k<l}R_k^{\sigma(i)},\]
    where $R_k^1=R_k$ and $R_k^0=R_k^c$.
    Then 
    \[(M,\mathcal{X})\models\fa x\exlt \sigma {2^l}x\in R_\sigma.\]
    Here we identify $\sigma$ with its code in $M$.
    Applying $\bd\Sigma_2^0$, we have
    \[(M,\mathcal{X})\models\exlt \sigma {2^l}\fa b\exlg x b x\in R_\sigma.\]
    Any $R_\sigma$ that is cofinal in $M$ will be $(R_k)_{k<l}$-cohesive.
\end{proof}

Collecting the results above, the strength of $\fincoh$ is determined in all cases except when $\RCA_0+\neg\bd\Sigma_2^0$ holds. Yokoyama~[personal communication] asked whether $\fincoh$ is actually equivalent to $\coh$ over $\RCA_0+\neg\bd\Sigma_2^0$.

\begin{question}[Yokoyama]
    Does $\RCA_0+\neg\bd\Sigma_2^0+\fincoh$ imply $\coh$?
\end{question}

At present, we are unable to answer this question.
Nevertheless, the following proposition demonstrates that if the answer is negative, then the first-order part of any model that witnesses the separation must satisfy a certain arithmetic principle lying strictly between $\ind\Sigma_1$ and $\bd\Sigma_2$.
This suggests the subtlety of the question.

The assumption in Proposition~\ref{prop:open} may be viewed as the failure of the surjective pigeonhole principle for $\Sigma_2^0$-definable partial functions.
Notice that the model constructed in Theorem~\ref{thm:rcanotprovefinsep} satisfies this assumption.
\begin{proposition}\label{prop:open}
    Let $(M,\mathcal{X})\models\RCA_0+\fincoh$.
    If there exist some $b\in M$ and a binary function $g\colon M\times M\to M$ in $\mathcal{X}$ such that for any $i\in M$, there is some $j<b$ satisfying $\lim_{s\to\infty} g(j,s)=i$, then $(M,\mathcal{X})\models\coh$.
\end{proposition}
\begin{proof}
    Let $R=(R_i)_{i\in M}\in\mathcal{X}$ be an infinite sequence of sets. We define a finite sequence of sets $S=(S_j)_{j<b}$ in $\mathcal{X}$ by 
    \[x\in S_j\iff x\in R_{g(j,x)}.\]
    By $\fincoh$ in $M$, there is some $H\in\mathcal{X}$ that is $S$-cohesive. We claim that $H$ is also $R$-cohesive.
    For any $i\in M$, by assumption, there are some $j<b$ and $s_0\in M$ such that $\falg s {s_0} g(j,s)=i$.
    By the definition of $S$,
    \[(M,\mathcal{X})\models\falg x {s_0}(x\in S_j\leftrightarrow x\in R_i),\]
    So, $H\subseteq^*R_i$ or $H\subseteq^*R_i^c$ as $H$ is $S$-cohesive.
\end{proof}
Another potentially interesting question concerns the following issue: while $\coh$ implies $\ind\Sigma_1^0$ over $\RCA_0^*$, it remains unknown whether $\sep$, which is equivalent to $\coh$ over $\RCA_0+\bd\Sigma_2^0$, is consistent with $\RCA_0^*+\neg\ind\Sigma_1^0$.
\begin{question}
    Does $\sep$ imply $\ind\Sigma_1^0$ over $\RCA_0^*$?
\end{question}
We remark that combining the technique of ~\cite{art:weakercousins} and Theorem~\ref{thm:rcanotprovefinsep}, one can show that $\finsep$ is already not $\forall \Pi_5^0$-conservative over $\RCA_0^*$. 
\section*{Acknowledgement}
The author's research was partially supported by the Singapore Ministry of Education grant MOE-000538-01, as well as by the NUS grants WBS A-0008494-00-00 and A-0008454-00-00. This work is contained in the author's Ph.D. thesis in the National University of Singapore. I would like to thank my two supervisors during my Ph.D, Tin Lok Wong and Yue Yang, for their guidance and encouragement. 
I would also like to thank Leszek A. Ko\l{}odziejczyk for careful proofreading and helpful suggestions.
I am also indebted to David R. Belanger and Keita Yokoyama for helpful discussions and suggestions.
\bibliographystyle{amsplain}
\bibliography{end}

\end{document}